\documentclass[11pt,reqno]{amsart}
\usepackage{amsthm, amsmath, amssymb, amscd, latexsym, multicol, verbatim, enumerate, }
\usepackage[usenames]{color}
\usepackage{hyperref}
\advance\textwidth by 1.2in \advance\oddsidemargin by -.6in \advance\evensidemargin by -.6in \parskip=.1cm
\theoremstyle{definition} 
\newtheorem{cor}{Corollary}
\newtheorem{lem}{Lemma}
\newtheorem{prop}{Proposition}
\newtheorem{thm}{Theorem}
\newtheorem{rem}{Remark}
\newtheorem{defn}{Definition}
 \newtheorem*{thm*}{Theorem}
\newtheorem*{coro*}{Corollary}

\newtheorem{exam}{Example}

\newcounter{cnt}
 \makeatletter
\def\mydggeometry{\makeatletter\dg@YGRID=1\dg@XGRID=20\unitlength=0.003pt\makeatother}
\makeatother \theoremstyle{remark}
\numberwithin{equation}{section}
\newcommand{\nc}{\newcommand}
\nc{\lie}[1]{\mathfrak{#1}} \nc\bp{\mathbf{p}} \nc\bq{\mathbf{q}} \nc\mD{\mathbb{D}} \nc\bs{\mathbf{s}} \nc\bt{\mathbf{t}} \nc\bz{\mathbb{Z}} \nc\bc{\mathbb C} 

\begin{document}
\author[Fourier]{Ghislain Fourier}
\address{Mathematisches Institut, Universit\"at Bonn}
\address{School of Mathematics and Statistics, University of Glasgow}
\email{ghislain.fourier@glasgow.ac.uk}

\thanks{DFG priority program 1388}

\date{\today}

\title[PBW graded Demazure for triangular elements]{PBW-degenerated Demazure modules and Schubert varieties for triangular elements}
\begin{abstract}
We study certain faces of the normal polytope introduced by Feigin, Littelmann and the author whose lattice points parametrize a monomial basis of the PBW-degenerated of simple modules for $\mathfrak{sl}_{n+1}$. We show that lattice points in these faces parametrize monomial bases of PBW-degenerated Demazure modules associated to Weyl group elements satisfying a certain closure property, for example Kempf elements.\\
These faces are again normal polytopes and their Minkowski sum is compatible with tensor products, which implies that we obtain flat degenerations of the corresponding Schubert varieties to PBW degenerated and toric varieties.
\end{abstract}
\maketitle

\section*{Introduction}
Let $\lie n$ be a finite-dimensional complex Lie algebra, then the PBW filtration on $U(\lie n)$ is defined
\[
U(\lie n)_s = \langle x_{i_1} \cdots x_{i_\ell} \; | \; x_{i_j} \in \lie n \, ; \, \ell \leq s \rangle_{\bc}.
\]
One has an induced filtration on any cyclic $\lie n$-module $M$. The associated graded module $M^a$ is a module for the abelianized Lie algebra $\lie n^a$ and the commutative algebra $S(\lie n)$. Let us consider the special case where $\lie g$ is a simple complex Lie algebra with triangular decomposition $\lie g = \lie n^+ \oplus \lie h \oplus  \lie n$ and $M = V(\lambda) = U(\lie n).v_\lambda$, a simple finite-dimensional highest weight module for $\lie g$. Then $V(\lambda)^a$ is a cyclic module for the deformed Lie algebra $\lie b \oplus \lie n^a$.\\
These graded modules $V(\lambda)^a$ have been studied under various aspects quite a lot in recent years \cite{Fei12, G11, FFoL11a, FFoL11b, FFoL13, FFoL13a, CF13, Fou14, BD14}. For example in \cite{FFoL11a} the annihilating ideal for $V(\lambda)^a$ as a $S(\lie n)$-module has been computed in the case $\lie{sl}_{n+1}$.  Moreover a normal polytope $P(\lambda)$ parametrizing a monomial bases of $V(\lambda)^a$ has been provided, e.g. if $N$ is the number of positive roots, $R^+$, then $P(\lambda) \subset \mathbb{R}_{\geq 0}^{N}$ and the inequalities defining this polytope can be described by using Dyck paths. The same has been done for the symplectic Lie algebra in \cite{FFoL11b} and in \cite{G11, BD14} in various other cases. The Dyck paths in our context: descending sequences of positive roots, with respect to the partial order $\leq$, introduce in \cite{FFoL11a}, see \eqref{ord-root1}. \\
In this paper we will extend the study of PBW filtrations to Demazure modules for $\lie {sl}_{n+1}$, e.g. submodules generated by the action of the Borel subalgebra through an extremal weight vector $V_w(\lambda) := U(\lie b).v_{w \lambda} \subset V(\lambda)$. A first step into this study has been taken in \cite{CF13}, where the maximal non-zero degree of such a PBW graded Demazure module has been computed. The maximal PBW degree of $V(\lambda)^a$ has been computed for all types of simple finite-dimensional Lie algebras in \cite{BBCF14}. Associated graded affine Demazure modules or Weyl modules have been studied in \cite{CO13, FM14}, while we will focus on Demazure modules for finite-dimensional Lie algebras.\\
For our study we will consider the following faces of the polytope $P(\lambda)$: Let $A \subset R^+$, a subset of the positive roots. We study the face $P_A(\lambda)$ of $P(\lambda)$ defined by setting all coordinates $s_\alpha = 0$ if $\alpha \notin A$, and denote $S_A(\lambda)$ the set of lattice points in $P_A(\lambda)$.\\
We introduce \textit{triangular} subsets $A \subset R^+$: \\
Let $\beta_1, \beta_2 \in R^+$ and let $\gamma(\beta_1, \beta_2)$ be the minimal positive root, such that $\gamma(\beta_1, \beta_2) \trianglerighteq \beta_1, \beta_2$ (here $\trianglelefteq$ is the usual order on positive roots: $\beta_1 \trianglelefteq \beta_2 \Leftrightarrow \beta_2 - \beta_1$ is a positive root). Then $A$ is \textit{triangular} if for all $\beta_1, \beta_2 \in A$ such that $\operatorname{supp } \beta_1 \cup \operatorname{supp }\beta_2$ is connected: 
\[\gamma(\beta_1, \beta_2), \beta_1 + \beta_2 - \gamma(\beta_1, \beta_2) \in A.\]
Further we say a Weyl group element $w \in S_{n+1}$ is called \textit{triangular}  if  $w^{-1}(R^-) \cap R^+$ is triangular. Kempf elements are a proper subset of the set of triangular elements (Lemma~\ref{kempf-prop}, Remark~\ref{kempf-rem}), for other interesting examples see Remark~\ref{exam-tri}.\\
Let $A \subset R^+$ and denote  $\lie n_A \subset \lie n$ the Lie subalgebra generated by the negative root vectors of roots in $A$, further set $V_A(\lambda) := U(\lie n_A).v_\lambda \subset V(\lambda)$.
\begin{thm*} Let $A \subset R^+$ be triangular, then:
\begin{enumerate}
\item $P_A(\lambda)$ is a normal polytope.
\item For all $\lambda, \mu \in P^+$: $S_A(\lambda) + S_A(\mu) = S_A(\lambda + \mu)$.
\item $S_A(\lambda)$ parametrizes a monomial basis of the PBW graded module $(V_A(\lambda))^a$.
\item Suppose $A = w^{-1}(R^-) \cap R^+$ for some $w \in S_{n+1}$, then $S_A(\lambda)$ parametrizes a monomial basis of the PBW graded Demazure module $V_w(\lambda)^a := (U(\lie b).v_{w\lambda})^a$.
\end{enumerate}
\end{thm*}
We should remark here, that if $w$ is not triangular, then the lattice points in $S_A(\lambda)$ do not parametrize a basis of $V_w(\lambda)^a$. In fact, we still obtain a linear independent subset but this is not a spanning set of $V_w(\lambda)$ (if $\lambda$ is regular). Hence we can not generalizes further, since the conditions are necessary if we want to consider faces of $P(\lambda)$.\\ 
The last point provides a non-recursive character formula for Demazure modules (for triangular elements) in terms of the lattice points of the face of $P(\lambda)$ (Corollary~\ref{dem-char}). \\
Another approach, for arbitrary Weyl group elements but with restriction to multiples of fundamental weights has been provided in \cite{BF14} using another polytope description similar to one provided by \cite{BD14}. It is proved in \cite{BF14} that the there is a marked poset (see \cite{Sta86, ABS11}) such that the polytope describing a monomial basis of the PBW graded Demazure module is a face of the marked chain polytope while the marked order polytope is a Kogan face of the well-known Gelfand-Tsetlin polytope (\cite{KST12}, \cite{Ko00}). It would be interesting to explore this connection in our context of arbitrary highest weights, this will be part of future research.\\
The notion of a favourable module $M$ for a unipotent complex algebraic group $\mathbb{U}$ has been introduced in \cite{FFoL13a}. Certain interesting properties such as flat toric degenerations of the associated flag varieties (see for more details Section~\ref{sec-five} or \cite{FFoL13a}) of a favourable module are governed by combinatorial properties of a convex polytope. A cyclic module $M$ is called favourable if there is a polytope whose lattice points $S(M)$ parametrize a basis of $\operatorname{gr} M$, where the grading is induced from a total ordering on a basis of $\lie u$ (the Lie algebra associated with $\mathbb{U}$) and an induced homogeneous lexicographic ordering on monomials in this basis. Further, the lattice points in the $n$-th dilation of $S(M)$ (the $n$-th Minkowski sum) parametrize a basis of the Cartan component of the $n$-times tensor product of $M$.
\begin{coro*} Let $\lambda \in P^+$, $ \subset R^+$ triangular, then $V_A(\lambda)$ is a favourable module. Further, if there exists $w \in W$ such that  $A = w^{-1}(R^-) \cap R^+$, then $V_w(\lambda)$ is a favourable module.
\end{coro*}
Thus we obtain for all triangular $w \in W$ a flat degeneration of the Schubert variety $X_w = \overline{B.[v_{w(\lambda)}]} \subset \mathbb{P}(V_w(\lambda))$ (and especially of the Kempf variety) into a PBW degenerated and further into a toric variety. These varieties are projectively normal and arithmetically Cohen-Macaulay varieties.  Note that it has been shown in \cite{FFoL13a} (for the special case $w = w_0$) that this toric variety is, in general, not isomorphic to the toric varieties constructed in \cite{GL96}. One can easily verify using the program \textit{polymake} (\cite{GJ00}) that the associated normal fans have different numbers of lower dimensional faces.\\

The paper is organized as follows: In Section~\ref{sec-one} we give basic definitions and introduce the terms triangular subset and triangular Weyl group element, show that Kempf elements are triangular. In Section~\ref{sec-two} we define the polytopes $P_A(\lambda)$ and provide a connection to marked posets. Section~\ref{sec-three} is about the representation theory, we recall the PBW graded modules and state the main results of the paper, while the proofs are in Section~\ref{sec-four}. Section~\ref{sec-five} is dedicated to the applications to degenerated Schubert varieties.\\

\noindent
\textbf{Acknowledgments} The author is supported by the project ''Shuffles and Schur positivity'' within DFG priority program 1388 ''Representation Theory''. 
\section{Definitions and basics}\label{sec-one}
\subsection{Preliminaries}
Let $\lie g = \lie{sl}_{n+1}$ with the standard triangular decomposition 
\[
\lie g = \lie b^+ \oplus \lie n^- = \lie n^+ \oplus \lie h \oplus \lie n^-.
\]
We denote the set of roots $R$, the set of positive roots $R^+$. Simple roots are denoted $\alpha_i, i = 1 \ldots, n$, and so we can write any positive root as 
\[\alpha_{i,j} :=\alpha_ i + \alpha_{i+1} + \ldots + \alpha_j \; ; \; \alpha_{i,i} := \alpha_ i.\] 
We introduce a partial order on $R^+:$
\begin{eqnarray}\label{ord-root1}
\alpha_{i_1, j_1} \geq \alpha_{i_2, j_2} :\Leftrightarrow i_1 \leq i_2 \text{ and } j_1 \leq j_2.
\end{eqnarray}
Note that this is not the usual partial order on roots. The cover relation in the case $\lie{sl}_4$ is: 
\[
\alpha_1 > \alpha_1 + \alpha_2 > \alpha_1 + \alpha_2 + \alpha_3 > \alpha_2 + \alpha_3 > \alpha_3\; ; \; \alpha_1 + \alpha_2 > \alpha_2 > \alpha_2 + \alpha_3.
\]
For any $\alpha_{i,j} \in R^+$, fix a $\lie{sl}_2$-triple $\{ e_{\alpha_{i,j}}, h_{\alpha_{i,j}} = [e_{\alpha_{i,j}}, f_{\alpha_{i,j}}], f_{\alpha_{i,j}} \}$, where ${e_{\alpha_{i,j}} = E_{i,j+1}, f_{\alpha_{i,j}} = E_{i+1, j}}$.\\
The lattice of integral weights is denoted $P$, and the fundamental weights $\omega_i ,i = 1 , \ldots, n$, $P^+ = \bigoplus_{i \in I} \bz_{\geq 0} \omega_i$ denotes the set of dominant integral weights.\\
We denote the universal enveloping algebra by $U(\lie g)$, the famous PBW theorem implies that the set of monomials in an ordered basis of $\lie{sl}_{n+1}$ is basis of $U(\lie g)$.


\subsection{Weyl group combinatorics}
We denote the Weyl group of $\lie{sl}_{n+1}$ by $W$. Then $W$ is generated by the reflections at simple roots, namely $s_i, i = 1, \ldots, n$, and $W$ is isomorphic to the group of permutations of $\{1, \ldots, n+1\}$.\\
\begin{defn}
We say $w \in W $ is \textit{triangular} if $w$ satisfies the following property:\\
Let $i < k \leq j < \ell$ and suppose $w(i) > w(j), w(k) > w(\ell)$, then
\begin{equation}
w(i) > w(\ell) \, , \, w(k) \geq w(j)
\label{eq:def-rec}
\end{equation}
\end{defn}

The study of triangular subsets might be motivated by the following examples of triangular Weyl group elements.

\begin{exam}\label{exam-tri} Denote $s_i := (i, i+1) \in S_{n+1}$, then the following elements in $S_{n+1}$ are triangular:
\begin{itemize}
\item $w = \left(s_{i+k}s_{i+k+1} \cdots s_{i + 2k}\right)\left(s_{i+k-1} \cdots s_{i+2(k-1)}\right)\cdots \left(s_{i+1} s_{i+2}\right)s_i
$
for any $1 \leq i,k \leq n$ such that $i + 2k \leq n$.
\item $w = \left(s_{i+j} \cdots s_i \right)\left(s_{n} \cdots s_{2}\right) \cdots \left(s_n s_{n-1}\right) s_n$ for any $i + j \leq n$.
\item Any longest element in a subgroup generated by simple reflections.
\item Any Kempf element is triangular (see Definition~\ref{kempf} and Proposition~\ref{kempf-prop}).
\item Not every Weyl group element is triangular, the ''smallest'' element which is not triangular is $s_1s_3s_2 \in S_4$.
\end{itemize}
Especially the first one appeared before in the context of PBW graded modules. Namely it is shown in \cite{CLL14}, that the PBW graded module $V(\lambda)^a$ is in fact isomorphic, as a module for $\lie b \oplus \lie n^a$ to the Demazure $V_w(\tilde{\Lambda})$, where $V(\tilde{\lambda})$ a simple module $\lie{sl}_{2n}$-module of highest weight $\tilde{\lambda}$. We see that we could apply our degeneration to this module $V_w(\tilde{\lambda})$, but of course, this module is already PBW graded.
\end{exam}
To each $w \in W$ we associate the following subsets of roots:
\[
R^+_w := \{ \alpha\in R^+ \, | \, w^{-1}(\alpha) \in R^- \} \subset R^+ \; ; \;  R^-_w := w^{-1}(R^+) \cap R^-, 
\]
then for  $i \leq j$:
\begin{equation}
-\alpha_{i,j} \in R^-_w \Leftrightarrow w(i) > w(j+1).
\label{eq:roots}
\end{equation}


\subsection{Triangular subsets}
We translate the term \textit{triangular Weyl group element} to subsets of roots $R_w^-$. We start with a more general construction and see that the case, we are interested in, is covered by this construction.
\begin{defn}
A subset $A \subset R^+$ is called \textit{triangular} if and only if for all $ \alpha_{i_1, j_1} > \alpha_{i_2, j_2}$ with $ i_2 \leq j_1 + 1$:
\[ \alpha_{i_1,j_1}, \alpha_{i_2, j_2} \in A  \Rightarrow \alpha_{i_1, j_2} \in A \text{ and if } i_2 \leq j_1:  \alpha_{i_2, j_1} \in A.
\]
\end{defn}
\begin{rem} Another way of introducing is: \\
Let $\beta_1, \beta_2 \in R^+$ and we introduce yet another partial order on $R^+$ (the usual one):
\[
\beta_1  \trianglelefteq \beta_2 :\Leftrightarrow \beta_2 - \beta_1 \in R^+.
\]
We set further $\operatorname{supp } \alpha_{i,j} := \{i, \ldots,j \}$. \\
Suppose $\beta_1, \beta_2 \in A$ and $\operatorname{supp } \beta_1 \cup \operatorname{supp} \beta_2$ is connected. Then $A$ is called \textit{triangular} if and only if the minimal $\gamma \in R^+$ s.t. $\beta_1, \beta_2 \trianglelefteq \gamma$, is in $A$ and, if it exists, the maximal $\delta \in R^+$, s.t. $\beta_1, \beta_2 \trianglerighteq \delta$, is also in $A$ (certainly $\delta = \beta_1 + \beta_2 - \gamma$).
\end{rem}
\begin{exam}\label{examp1}
Let us explain the term \textit{triangular} with the following picture. The set of positive roots can be arrange in the triangle of lower triangular matrices:
$$
\left( 
\begin{array}{cccccccc}
\alpha_{1} & & & & & \\
\alpha_{1,2} & \alpha_{2} & & & & \\
\alpha_{1,3} & \alpha_{2,3} & \alpha_{3}& & & \\
\vdots & \vdots & \vdots & & & \\
\alpha_{1,n} & \alpha_{2,n}&  \alpha_{3,n} & \ldots &\alpha_{n}\\
\end{array}
\right)
$$
A subset $A \subset R^+$ is called \textit{triangular} if for any pair of roots 
\[\alpha_{i_1, j_1} >  \alpha_{i_2, j_2} \in A, \text{ with } i_2 \leq j_1 + 1 ,\]
the root $ \alpha_{i_1, j_2}$ in the triangle
$$
\left( 
\begin{array}{ccccccccc}
& & & & & & \\
& \alpha_{i_1, j_1} & & & &  & \\
&  & & & & & \\
& \alpha_{i_1,j_2} &  &  & &\alpha_{i_2, j_2}  &\\
\end{array}
\right)
$$
is also in $A$. Further if $i_2 \leq j_1 $, then the root $ \alpha_{i_2, j_1}$ in the triangle
$$
\left( 
\begin{array}{ccccccccc}
& & & & & & \\
& \alpha_{i_1, j_1} & & & \alpha_{i_2,j_1}  & &\\
&  & & & & & \\
&  &  & & \alpha_{i_2, j_2}  &\\
\end{array}
\right)
$$
is also in $A$.
\end{exam}

To each subset $A \subset R^+$ we can associate the subspace generated by the corresponding root vectors.
\[
\lie n_A := \langle f_\alpha \, | \, \alpha \in A \rangle_{\bc} \subseteq \lie n
\]
\begin{cor} Suppose $A \subset R^+$ is triangular, then  $\lie n_A$ is a Lie subalgebra of $\lie n$.
\end{cor}
\begin{proof}
We have to check that $\lie n_A$ is closed for the adjoint action. Suppose $[ f_\alpha, f_\beta]  \neq 0 $, then $[ f_\alpha, f_\beta] \in \bc f_{\alpha + \beta}$. Triangular implies that if $\alpha, \beta \in A$ and $\alpha + \beta \in R^+$, then $\alpha + \beta \in A$.
\end{proof}

\begin{prop}\label{tritri}
$w \in W$ is triangular if and only if $-R_w^-$ is triangular.
\end{prop}
\begin{proof}
Let $w \in W$. Let $i \leq k \leq j+1 < \ell + 1$ and suppose $\alpha_{i,j}, \alpha_{k, \ell} \in -R_w^-$. Via the description \eqref{eq:roots} we see that this implies
\[
w(i) > w(j+1)\,  , \, w(k) > w(\ell + 1).
\]
Since $w$ is triangular we have by definition \eqref{eq:def-rec}:
\[
w(i) > w(\ell+1), w(k) \geq w(j+1).
\]
But again with \eqref{eq:roots} this is equivalent to
\[
\alpha_{i,\ell} \in -R_w^- \text{ and if } k \leq j : \alpha_{k,j} \in -R^-_w
\]
\end{proof}


\subsection{Kempf elements}
\begin{defn}[\cite{HL85}]\label{kempf} Let $w \in S_{n+1}$, say $w = w_1 w_2 \cdots w_n$ where $w_i$ is a right-end segment of $u_i = s_n \cdots s_{i+1}s_i$. $w$ is called Kempf element if 
\[
\ell(w_i) \leq \ell(w_{i+1}) + 1 \text{ whenever } w_{i+1} < u_{i+1} , 1 \leq i \leq n-1.
\]
\end{defn}

\begin{lem}\label{kempf-prop}
Let $w \in S_{n+1}$ be Kempf, then $w$ is triangular.
\end{lem}
\begin{proof}
Let $w$ be Kempf, $w = w_1 \cdots w_n$, where $w_i = (s_{\ell_i} \cdots s_i), n \geq \ell_i \geq i$. Since $w$ is Kempf we have $l(w_i) \leq l(w_{i+1})+1$ if $\ell_{i+1} \neq n$ and hence $\ell_i \leq \ell_{i+1}$. We will compute $w^{-1}(R^-)$ and show that this set is triangular. First of all remark that $w_0^{-1}(R^-) = R^+$ for $w_0$ the longest Weyl group element. The set of all positive roots forms a triangle $\{ \alpha_{i,j} \, | \, 1 \leq i \leq n, i \leq j \leq n \}$ (see Example~\ref{examp1}), where the $i$-th column is $\{ \alpha_{i,i}, \ldots, \alpha_{i,n} \}$. \\
For the given $w$ be will compute the positive roots which are not in $w^{-1}(R^-)\cap R^+$ and show that they form a useful pattern. We consider 
\[
S = \bigcup_{i=1}^{n} \{ \alpha_{k,i+j} \, | \,  1 \leq k \leq i, 0 \leq j \leq \ell_{i+1} - \ell_i - 1 \}.
\]
For each $i$ we have a block from column $1$ to column $i$ of height $\ell_{i+1} - \ell_i$, where the upper right corner is the root $\alpha_{i,i}$. Note that if $\ell_{i+1} = \ell_i$ this block is empty. We can assume that none of the $w_i$ is trivial (equal to the identity), else the problem would be split into two independent (smaller problems).\\
Claim:
\[
w^{-1}(R^-)\cap R^+ = R^+ \setminus S.
\]
It is easy to verify that both sets have the same cardinality: $\ell(w) = |R^+| - |S|$. So it remains to show that if $\alpha \in R^+ \setminus S$, then $\alpha \in w^{-1}(R^-)\cap R^+$. \\
Let $\alpha_{i,j} \in R^+ \setminus S$, we have to show that $w(\alpha_{i,j}) \in R^-$. Now, since $\alpha_{i,j} \in R^+ \setminus S$, we have $\alpha_{k,j} \in R^+ \setminus S$ for all $i \leq k \leq j$. Further $w_{j+2} \cdots w_n (\alpha_{i,j}) = \alpha_{i,j}$ and $w_{j+1}(\alpha_{i,j}) = \alpha_{i,\ell_{j+1}}$.\\
We have $\alpha_{k,j} \in R^+ \setminus S$ for all $i \leq k \leq j$, this implies that $\ell_j = \ell_{j+1}$, else we would have $\alpha_{j,j} \in S$ which is a contradiction. This implies that 
\[
w_j(\alpha_{i,\ell_{j+1}}) = \alpha_{i, \ell_{j} - 1}.
\]
Again, since $\alpha_{k,j} \in R^+ \setminus S$ for all $i \leq k \leq j$, we have $\ell_j -1 \leq \ell_{j-1} \leq \ell_j$, this implies that  
\[
w_{j-1}(\alpha_{i,\ell_{j}-1}) = \alpha_{i, \ell_{j} - 2}.
\]
Iterating this gives
\[
w_{i+1} \cdots w_{j+1}(\alpha_{i,j}) = \alpha_{i, \ell_j - (j-i)}
\]
and $\ell_j - (j-i) \leq \ell_i$. Now 
\[
w_i(\alpha_{i, \ell_j - (j-i)}) = -\alpha_{\ell_j - (j-i),  \ell_i}.
\]
Now the claim follows since $w_1 \cdots w_{i-1}(-\alpha_{\ell_j - (j-i),  \ell_i}) \in R^-$.\\

\noindent
With Proposition~\ref{tritri}, it remains to show that $R^+ \setminus S$ is triangular. \\
Let 
\[
\alpha_{i_1, j_1}, \alpha_{i_2, j_2} \in R^+\setminus S  \, , \, i_1 < i_2\,,\, j_1 < j_2\, , \, i_2 \leq j_1 + 1,
\]
and suppose $\alpha_{i_1, j_2} \in S$. \\
Let $k$ be maximal such that $\alpha_{i_1 + k, j_2} \in S$, by construction of $S$, and since $\alpha_{i_1, j_2}$ is not a simple root: $k \geq 1$,  $i_1 + k < i_2$. Then for all $i_1 + k \leq l \leq j_2$: $\alpha_{i_1 + k, l} \in S$. This implies that $j_1 < i_1 + k < i_2$ and so $j_1 + 2 \leq i_2$. This is a contradiction to $i_2 \leq j_1 + 1$.\\
Further, if $i_2 \leq j_1$ and $\alpha_{i_2, j_1} \in S$, then by definition of $S$: $\alpha_{k, j_1} \in S$ for all $1 \leq k \leq i_2$, so especially $\alpha_{i_1, j_1}$, which is again a contradiction.
\end{proof}
\begin{rem}\label{kempf-rem}
The converse of the proposition is not true, the set of Kempf elements is a proper subset of the set of triangular elements. For example in the $S_4$-case: $w = s_2 s_3 s_1$ is triangular (since $-R_w^- = \{ \alpha_1, \alpha_3, \alpha_1 + \alpha_2 + \alpha_3 \}$) but it is not Kempf.
\end{rem}

\section{Posets and polytopes}\label{sec-two}

\subsection{Dyck path}
We recall the definition of a Dyck path for $\lie{sl}_{n+1}$ due to \cite{FFoL11a}:\\
A Dyck path is a sequence of positive roots $\mathbf{p} = ( \beta_1, \ldots, \beta_s)$
such that $\beta_1$ and $\beta_s$ are simple roots and if $\beta_k = \alpha_{i,j}$, then $\beta_{k+1} = \alpha_{i+1, j}$ or $\alpha_{i, j+1}$, especially $\beta_k > \beta_{k+1}$. We denote $\mathbb{D} \subset \mathcal{P}(R^+)$ the set of all Dyck paths for $\lie{sl}_{n+1}$.
 For $\mathbf{p} \in \mathbb{D}$, say $\bp = ( \alpha_{i_1, j_1}, \ldots, \alpha_{i_s, j_s} )$, we define the \textit{base root of} $\mathbf{p}$:
\[
\beta_\mathbf{p} := \alpha_{i_1, j_s}.
\]

We turn to the case of the current paper. Let $ A \subset R^+$ be any subset, then we define paths and Dyck paths for $A$  by restricting Dyck paths for $R^+$:
\begin{defn}
We say $\mathbf{p}  \in \mathcal{P}(A)$ is a \textit{path for} $A$ if there exists a $\mathbf{q} \in \mathbb{D}$ such that $\mathbf{p} = \mathbf{q} \cap A$, e.g. we may assume that there exist $i_1 \leq \ldots \leq i_s, j_1 \leq \ldots \leq j_s$ such that
\[
\mathbf{p} = (\alpha_{i_1, j_1}, \ldots, \alpha_{i_s, j_s}) \in \mathcal{P}(A).
\]
A path $\mathbf{p}$ is called \textit{Dyck path for} $A$ if for all $1 \leq k, \ell \leq s$:
\[
i_k \leq j_\ell \Rightarrow \alpha_{i_k , j_\ell} \in A.
\]
Let $\mD_A$ denote the set of all Dyck path of $A$.
\end{defn}
We define the base root of a Dyck path $\mathbf{p} = (\alpha_{i_1, j_1}, \ldots, \alpha_{i_s, j_s} ) \in \mathbb{D}_A$ similar to Dyck paths in $\mD$: 
\[
\beta_\mathbf{p} := \alpha_{i_1,  j_s} \in A.
\]
\begin{rem}\label{tri-paths} For the sake of notation and legibility, we are a bit sloppy here by assuming that for all path we have $j_s + 1 \geq i_{s+1}$, hence the union of the support of successors is connected. If this union is not connected, the path would split into two path, supported on unrelated set of roots. This path would be redundant for the definition of the polytope.\\
In this sense we can say, that if $A\subset R^+$ is triangular then all paths of $A$ are Dyck path.
\end{rem}


\subsection{Polytopes}
Again following  \cite{FFoL11a}, we define polytopes using Dyck paths. Let $\lambda \in P^+$, $A \subset R^+$ be any subset, the polytopes $ P(\lambda) \subset \mathbb R_{\geq 0}^{\sharp R^+}, P_A(\lambda) \subset \mathbb R_{\geq 0}^{\sharp A}$ are defined as follows:
\[
P(\lambda) = \{ (s_\alpha)_{\alpha \in R^+} \in (\mathbb R_{\geq 0})^{\sharp R^+} \, | \,\forall \, \mathbf{p} \in \mathbb{D}: \sum_{\alpha \in \mathbf{p}} s_\alpha \leq \lambda (h_{\beta_\mathbf{p}}) \}
\]
and
\[
P_A(\lambda) = \{ (s_\alpha)_{\alpha \in A}  \in  (\mathbb R_{\geq 0})^{\sharp A} \, | \,\forall \, \mathbf{p} \in \mathbb{D}_A: \sum_{\alpha \in \mathbf{p}} s_\alpha \leq \lambda (h_{\beta_\mathbf{p}}) \}
\]
We denote the set of lattice points in $P(\lambda)$ (resp. $P_A(\lambda)$):
\[
S(\lambda) = P(\lambda) \cap (\mathbb{Z}_{\geq 0})^{\sharp R^+} \, , \, S_A(\lambda) = P_A(\lambda) \cap (\mathbb{Z}_{\geq 0})^{\sharp S}.
\]

The polytope $P(\lambda)$ has been defined in \cite{FFoL11a} (E.~Vinberg suggested this polytope in a conference talk without publication) and it has been proved in \cite{FFoL11a}:
\begin{thm*} $P(\lambda)$ is a normal polytope and further, for all $\lambda, \mu \in P^+$: $S(\lambda) + S(\mu) = S(\lambda + \mu)$ (the Minkowski sum). 
\end{thm*}
\begin{rem}\label{ffl}
It is important to notice that none of the inequalities is redundant, e.g. one obtains a strictly bigger polytope (for regular $\lambda \in P^+$) if one omits one of the paths.
\end{rem}


\subsection{Marked posets}

We recall certain marked polytopes associated to a marked poset (following \cite{ABS11}, see also \cite{Sta86}).
Given a finite poset $P = \{ x_p \, | \,  p \in P\}$ and a subset $M \subset P$ which contains at least all minimal and maximal elements of $P$, further fix $\lambda \in (\mathbb{Z}_{\geq 0})^{\sharp M}$ satisfying $\lambda_{m_1} \geq \lambda_{m_2}$ if $m_1 \geq m_2$. Then there are two interesting polytopes associated to the tuple $(P, M, \lambda)$, namely the \textit{marked order polytope}
\[
\mathcal{O}(P, M, \lambda) = \{ (x_p) \in \mathbb{R}^{P \setminus M} \, | \,  x_p \leq x_q, p < q ; \lambda_m, \leq x_p , m < p; x_p \leq \lambda_m, p < m \}
\]
and the \textit{marked chain polytope}
\[
\mathcal{C}(P, M, \lambda) =  \{ (x_p) \in \mathbb{R}_{\geq 0}^{P \setminus M} \, | \,  x_{p_1} + \ldots + x_{p_1} \leq \lambda_{m_1} - \lambda_{m_2}; m_2  < p_1 < \ldots < p_s < m_1 \}
\]
\begin{thm*}[\cite{ABS11, Sta86}]
Let $(P, M, \lambda)$ be a marked poset, then $\mathcal{O}(P, M, \lambda)$ and $\mathcal{C}(P, M, \lambda)$ have the same Ehrhart polynomial and hence the number of lattice points is equal.
\end{thm*}
Recall the partial order on $R^+$ in our context:
\[
\alpha_{i_1, j_1} \geq \alpha_{i_2, j_2} :\Leftrightarrow i_1 \leq i_2 \text{ and } j_1 \leq j_2.
\]
By adding a set of vertices $M = \{ a_1, \ldots, a_n, a_{n+1} \}$, with relations
\[ \alpha_{i+1} >  a_i > \alpha_i, i= 2, \ldots, n \text{  and } a_1 > \alpha_1, \alpha_n > a_{n+1},\] 
we extend this partial order to $R^+ \cup M$ and obtain a partial order on $A \cup M$ by restricting from $R^+ $ to $A$. Then $A \cup M$ has a unique maximal element, $a_1$ and a unique minimal element $a_{n+1}$, both contained in $M$.

\begin{lem}
Let $\lambda = \sum_{i} m_i \omega_i \in P^+$ be regular (e.g. $m_i \neq 0$ for all $i$) and $M$ be labeled as: $a_i = m_i+ \ldots + m_n$, $a_{n+1} = 0$. Then
$A\subset R^+$ is triangular if and only if $P_A(\lambda)$ is the marked chain polytope $\mathcal{C}(A\cup M, M, \lambda)$.
\end{lem} 
\begin{proof}
We have to show that the defining inequalities are the same for both polytopes. Certainly both polytopes are in $(\mathbb{R}_{\geq 0})^{\sharp A}$. \\
Let $\bp \in \mD_A$, then $\bp$ is a descending sequence of roots with respect to the partial order. So we can view $\bp$ as a chain $(\alpha_{i_1, j_1 }, \ldots, \alpha_{i_s, j_s})$ in the marked chain polytope. Let $a$ be the minimal and $b$ be the maximal marked vertex such that $a_i > \alpha > a_j$ for all $\alpha \in \bp$. Then $i = i_1, j = j_s +1$ and hence
\begin{eqnarray}\label{lem-1}
\sum_{\alpha \in \bp} s_\alpha \leq (m_{i_1} + \ldots + m_n - m_{j_s+1} - \ldots - m_n) = \lambda(h_{i_1, j_s}).
\end{eqnarray}
But $\alpha_{i_1, j_s}$ is the base root of $\bp$ and hence $\mathcal{C}(A\cup M, M, \lambda) \subset P_A(\lambda)$.\\
On the other hand, let $(a_{i_1} > \alpha_{i_1, j_1} > \ldots > \alpha_{i_s, j_s} > a_{j_s + 1})$ be a chain in $A \cup M$. We denote the path  $\bp = (  \alpha_{i_1, j_1}, \ldots,  \alpha_{i_s, j_s} )$ and we have to show that this is in fact a Dyck path for $A$. But with Remark~\ref{tri-paths} we know that for triangular $A$, all paths a re Dyck paths. Even more, the base root of $\bp = \alpha_{i_1, j_s}$ and then again \eqref{lem-1} shows that $P_A(\lambda) \subset \mathcal{C}(A\cup M, M,\lambda)$.\\
To show the if and only if part, notice that all inequalities induced from Dyck paths in $\mD_A$ are also inequalities induced from chains. So we are left to show that there is an inequality induced from a chain which is not induced from a Dyck path.\\
Suppose now $A$ is not triangular, then there exists $\bq \in \mD$ such that $\bp : = \bq \cap A$ is not a Dyck path but certainly a chain of descending roots $\alpha_{i_1, j_1} > \ldots > \alpha_{i_s, j_s}$, so we have the inequality in the marked chain polytope
\[
\sum_{\alpha \in \bp} s_\alpha \leq \lambda(h_{\beta_{\bp}})
\]
while this inequality is not satisfied in $P_A(\lambda)$ (since $\lambda$ is regular and none of the Dyck paths is redundant). 
\end{proof}

\begin{rem}
If $\lambda \in P^+$ is not regular, but $A \subset R^+$ is triangular, then the proof above shows that $P_A(\lambda)$ is still the marked chain polytope. But in this case the converse, if the marked chain polytope is equal to $P_A(\lambda)$ then $A$ is triangular, is not true in general.
\end{rem}

\begin{rem}
It would be interesting to study the corresponding marked order polytopes.  In \cite{BF14}, it has been proved that for fundamental weight and arbitrary $w \in W$, the marked order polytope (associated to a polytope describing a PBW graded basis), is in fact a Kogan face of the Gelfand-Tsetlin polytope. We expect such strong connection to Kogan faces of Gelfand-Tsetlin polytopes in more generality, especially for Kempf elements (see \cite{KST12} for more details on Kogan faces).
\end{rem}

\subsection{Faces}
Suppose  $A \subset R^+$ is triangular, then we can identify $P_A(\lambda)$ with a certain face of $P(\lambda)$, and identify the lattice points of $S_A(\lambda)$ with lattice points in $S(\lambda)$:
\begin{prop}\label{prop-embed}
 The map 
\[
\mathbf{s} = (s_\alpha)_{\alpha \in A} \mapsto \overline{\mathbf{s}} = (t_\alpha)_{\alpha \in R^+}, \text{ where } t_\alpha := \begin{cases} s_\alpha \text{ if } \alpha \in A \\ 0 \text{ else }\end{cases}
\]
is an embedding $P_A(\lambda) \hookrightarrow P(\lambda)$ which is surjective onto the face defined by $s_\alpha = 0$ for all $\alpha \notin A$.
\end{prop}

\begin{proof}
Let $\bs \in S_A(\lambda)$ and $\bp \in \mD$. We set $\bq = \bp \cap A$, then by definition and Remark~\ref{tri-paths} $\bq \in \mD_A$. Further $\lambda(h_{\beta_{\bq}}) \leq \lambda(h_{\beta_{\bp}})$ since $\lambda \in P^+$. Then using the map $\bs \mapsto \overline{\bs}$ we have
\[
\sum_{\alpha \in \bp} t_{\alpha} = \sum_{\alpha \in \bp \cap A} t_{\alpha} = \sum_{\alpha \in \bq} s_{\alpha} \leq \lambda(h_{\beta_{\bq}}) \leq \lambda(h_{\beta_{\bp}})
\]
which implies that $P_A(\lambda) \subset P(\lambda)$. On the other hand let $\bs \in P(\lambda)$ such that $s_\alpha =0 $ for all $\alpha \notin A$. Let $\bq \in \mD_A$, then by definition there exists $\bp \in \mD$ such that $\bq = \bp \cap A$. Even more, we can choose $\bp$ such that $\beta_{\bp} = \beta_{\bq} =: \beta$,  then
\[
\sum_{ \alpha \in \bq} s_\alpha = \sum_{\alpha \in \bp \cap A} s_\alpha = \sum_{\alpha \in \bp} s_\alpha \leq \lambda(h_{\beta}),
\]
and so the map is onto the face defined by $s_\alpha = 0$ for all $\alpha \notin A$.
\end{proof}

\begin{lem}\label{mink-sum} For $\lambda \in P^+$, $P_A(\lambda)$ is a normal convex lattice polytope.
Further if $\lambda, \mu \in P^+$, then $S_A(\lambda + \mu) = S_A(\lambda) + S_A(\mu)$.
\end{lem}
\begin{proof}
Let $\mathbf{s} \in S_A(\lambda + \mu) \subset S(\lambda + \mu)$ (Proposition~\ref{prop-embed}) then there exist (see Remark~\ref{ffl}) $\mathbf{t}^1 \in S(\lambda), \mathbf{t}^2 \in S(\mu)$ such that $\mathbf{s} = \mathbf{t}^1 + \mathbf{t}^2$. 
But this implies that $\mathbf{t}^1_\alpha = 0 =  \mathbf{t}^2_\alpha$ for all $\alpha \notin A$. So by Proposition~\ref{prop-embed}: $\mathbf{t}^1 \in S_A(\lambda), \mathbf{t}^2 \in S_A(\mu)$. This implies that $S_A(\lambda) + S_A(\mu) = S_A(\lambda +  \mu)$. Similarly one can show that $n S(\lambda) = S(n \lambda)$.T
\end{proof}


\section{Representation theory}\label{sec-three}
\subsection{Preliminaries}
We recall some notations and facts from representation theory. Let $V$ be a finite-dimensional $\lie {sl}_{n+1}$-module, then $V$ decomposes into weight spaces with respect to the $\lie h$-action
\[
V = \bigoplus_{ \tau \in P } V_\tau =  \bigoplus_{ \tau \in P }  \{ v \in V \; | \; h.v = \tau(h).v \text{ for all } h \in \lie h\}
\]
\noindent
$P^+$ parametrizes the simple finite-dimensional modules up to isomorphism. For $\lambda \in P^+$ we denote the simple, finite-dimensional $\lie {sl}_{n+1}$-module of highest weight $\lambda$ by $V(\lambda)$. Further we denote by $v_{\lambda}$ a highest weight vector of $V(\lambda)$, then $V(\lambda) = U(\lie n).v_{\lambda}$ and even more $V(\lambda)$ is isomorphic to the quotient of $U(\lie g)$ by the left ideal generated by 
\[
\lie n^+, h- \lambda(h), f_{\alpha}^{\lambda(h_\alpha) + 1}.
\]

$U(\lie g)$ as a left $\lie g$-module, and if $\beta_1, \ldots, \beta_s \in R^+$, then $f_{\beta_1} \cdots f_{\beta_s}$ is an $\lie h$ weight vector of weight $- \beta_1 -\ldots - \beta_s$.


\subsection{Demazure modules}
Let $w \in W$ and $\lambda \in P^+$. Classical results imply that the weight space of weight $w(\lambda)$ in $V(\lambda)$ is one-dimensional, we denote a generator of this line $v_{w(\lambda)}$. The Demazure module corresponding to $w$ and $\lambda$ is defined to be
\[
V_w(\lambda) := U(\lie b).v_{w(\lambda)}.
\]
Note, that this is not a $\lie g$-module but a cyclic $\lie b$ submodule in $V(\lambda)$. We have defining relations for  the annihilating ideals \cite{Jos85}, namely 
\[
 e_{\alpha}^{-w(\lambda)(h_{\alpha}) + 1} \, , \,  h - w(\lambda)(h) \, : \, \alpha \in R^+, h \in \lie h
\]
are in the annihilating ideal. The latter is equal to $e_\alpha^{-\lambda\left(h_{w^{-1}(\alpha)}\right) + 1}$, which implies that $e_\alpha.v_{w(\lambda)} = 0$ unless $w^{-1}(\alpha) \in R^-$.\\


\subsection{Modules for triangular sets}

Let $A \subset R^+$ be triangular, then we denote
\[
V_A(\lambda) := U(\lie n_A).v_\lambda \subset V(\lambda).
\]
Again, this is certainly not a $\lie g$-module but a cyclic $\lie n_A$-submodule in $V(\lambda)$. \\
On the other hand we have a natural action of the Weyl group on $\lie g$ and $V(\lambda)$, so we may consider
\[
w^{-1}(U(\lie n^+)).v_{w(\lambda)}w = w^{-1}(U(\lie n^+))w w^{-1}.v_{w(\lambda)}w.
\]
We denote $\lie n_w$ the subalgebra generated by $\{ f_\alpha \, | \, \alpha \in w^{-1}(R^-) \cap R^+ \}$, then $ w^{-1}(\lie n^+)w \subset \lie n_w \oplus \lie n^+$ and $ w^{-1}.v_{w(\lambda)}w = v_{\lambda}$. This implies that
\begin{eqnarray}\label{tri=dem}
 V_w(\lambda) = w(U(\lie n_w).v_{\lambda})w^{-1}.
\end{eqnarray}
In this sense, the Demazure module $V_w(\lambda)$ is conjugated to $V_A(\lambda)$ where $A = w^{-1}(R^-) \cap R^+$.


\subsection{PBW filtration and associated graded modules}

We recall here the PBW filtration. Let $\lie n$ be a finite-dimensional Lie algebra, then the PBW filtration on $U(\lie n)$ is defined as
\[
 U(\lie n)_s := \langle x_i \cdots x_{\ell} \, : \, x_i \in \lie n, \ell \leq s \rangle_{\mathbb{C}}.
\]
The associated graded algebra is commutative and the PBW theorem states that this algebra is isomorphic to $S(\lie n)$.\\

Let $M$ be a cyclic $\lie n$-module with generator $m$, so $M = U(\lie n).m$. The PBW filtration on $U(\lie n)$ induces a filtration on $M$:
\[
 M_s := U(\lie n)_s.m.
\]
We denote the associated PBW graded space $M^a$. Remark, that this is not a module for $U(\lie n)$ but for $S(\lie n)$.\\ 
The image of $m$ generates $M^a$ as a $S(\lie n)$-module. This implies that there exists an ideal $I_M \subset S(\lie n)$ such that $M^a \cong S(\lie n)/I_M$.\\

In this paper we will study the PBW graded spaces associated to $V_w(\lambda)$ and $V_A(\lambda)$, where $w$ (resp. $A$) is triangular, e.g.
\[
V_w(\lambda)^a = S(\lie n^+).v_{w(\lambda)} \; ; \; V_A(\lambda)^a = S(\lie n_A).v_\lambda.
\]


\subsection{Main statement}
To any point $\bs = (s_\alpha)_{\alpha \in A}  \in (\bz_{\geq 0})^{\sharp A}$ we associate a monomial
\[f^{\bs} = \prod_{\alpha \in A} f_{\alpha}^{s_\alpha} \in S(\lie n_A).\]
One of the main results in this paper is the following:
\begin{thm}\label{main-thm}
Let $A \subset R^+$ be triangular and $\lambda \in P^+$, then the set
\[
\left\{ f^{\bs}.v_\lambda \, | \, \bs \in S(\lambda) \right\}
\]
is a basis of $V_A(\lambda)^a$.
\end{thm}
\begin{proof}
The proof is split into two parts. First we show in Proposition~\ref{spaset} that the given set is a spanning set (for the vector space).
Then we  show that the set is linear independent (Corollary~\ref{linset}).
\end{proof}

\begin{cor}\label{dem-char}
By choosing an order in each factor $f^{\bs}$, the following is a basis of $V_A(\lambda)$:
\[
\left\{ f^{\bs}.v_\lambda \, | \, \mathbf{s} \in S(\lambda) \right\}
\]
\end{cor}


\subsection{Application to Demazure modules}
\begin{thm}
Let $\lambda \in P^+$ and $w \in W$ be triangular. Then the set
\[
 \left\{ \prod_{ \alpha \in A } (e_{w(\alpha) })^{s_\alpha}. v_{w(\lambda)} \, : \, \bs\in S_A(\lambda) \right\}
\]
is a basis of $V_w(\lambda)^a$ and also (by choosing an order in each factor) of $V_w(\lambda)$.
\end{thm}
\begin{proof}
This follows immediately from \eqref{tri=dem} and Theorem~\ref{main-thm}.
\end{proof}
So for a triangular $w$, a basis of the Demazure module $V_w(\lambda)$ is parametrized by the lattice points in the face of $P(\lambda)$ defined by setting all coordinates $x_\alpha$ to $0$ where $\alpha \notin w^{-1}(R^-) \cap R^+$.

\begin{cor}
Let $\lambda \in P^+$, $w \in W$ triangular, then the character of the Demazure module is (non-recursively) given by
\[
\operatorname{char } V_w(\lambda) = e^{w(\lambda)} \sum_{\bs \in S_A(\lambda)} e^{-w(\operatorname{wt} \bs)}
\]
where $\operatorname{wt} \bs:= \sum_{\alpha \in A} s_{\alpha} \alpha$.
\end{cor}

\section{Proofs}\label{sec-four}

\subsection{Spanning set}
We want to show that $V_A(\lambda)^a$ is spanned by the vectors
\[
 \{ \prod_{\alpha \in A} f_\alpha^{s_\alpha}.v\lambda \, : \, \mathbf{s} \in S_A(\lambda) \}
\]
Since, by definition
\[
 V_A(\lambda)^a = S(\lie n_A).v_\lambda,
\]
so we have to show that for any $\bt = (t_\alpha) \in (\bz_{\geq 0})^{\sharp A}$:
\[
f^{\bt}.v_\lambda \in \langle f^{\bs}.v_\lambda \; : \; \bs \in S_A(\lambda) \rangle_\bc.
\]
Recall the total order on $R^+$ from \cite{FFoL11a}:
$$
\alpha_{1,1} \prec \alpha_{1,2} \prec \ldots \prec \alpha_{1,n} \prec \alpha_{2,2} \prec \ldots \prec \alpha_{2,n} \prec \ldots \prec \alpha_{n-1, n-1} \prec \alpha_{n-1, n} \prec \alpha_{n,n}.
$$
We have an induced order on the root vectors $f_{\alpha_{i,j}}$ and denote $\prec$ the induced homogeneous lexicographic order on $S(\lie n^-)$. By restriction we obtain a total order $\prec$ on $S(\lie n_A)$. Then

\begin{lem}\label{straight}
 Let $\bt \in (\bz_{\geq 0})^{\sharp A}, \bt  \notin S_A(\lambda)$, then
\[
 f^{\mathbf{t}}.v_\lambda = \sum_{\mathbf{s} \prec \mathbf{t}} c_{\mathbf{s}} f^{\mathbf{s}}.v_\lambda,
\]
where $\bs \in (\bz_{\geq 0})^{\sharp A}$ and $c_{\mathbf{s}} \in \bc$.
\end{lem}
\begin{proof}
We can restrict to the case where $\mathbf{t}$ is supported on a Dyck path $\mathbf{p} = (\alpha_{i_1, j_1} , \ldots \alpha_{i_s, j_s})$ only. By abuse of notation we can assume that exist $ i_1 < \ldots < i_s, j_1 < \ldots < j_t$ such that $\bp$ is a totally ordered subset :
\begin{eqnarray}\label{rootA}
\mathbf{p} \subset \{ \alpha_{i_k, j_\ell} \, | \,  1 \leq k \leq s, 1 \leq \ell \leq t , i_k \leq j_\ell \} \subset A
\end{eqnarray}
and $\beta_\bp = \alpha_{i_1, j_t}$. \\
Let $\mathbf{t} = (t_{i_k,j_\ell})_{1 \leq k \leq s, 1 \leq \ell \leq t}$. Then by assumption $t_{p,q} = 0$ if $\alpha_{p,q} \notin \mathbf {p}$ and further
$$|\bp| = \sum_{k,\ell} t_{i_k, j_\ell}  >  \lambda(h_{\beta_\mathbf{p}}).$$ 
This implies that
\[
f_{\beta_\bp}^{|\bp|}.v_\lambda  = 0 \text{ in } V_A(\lambda) \subset V(\lambda) \text{ and hence } f_{\beta_\bp}^{|\bp|}.v_\lambda  = 0 \text{ in } V_A(\lambda)^a.
\]
We set
\[
t_{\circ, j_{\ell}} = \sum_{k=1}^{s} t_{i_k, j_{\ell}} \; , \; t_{i_k, \circ} = \sum_{\ell=1}^{t} t_{i_k, j_{\ell}}
\]
and consider the following expression in $V_A(\lambda)^a$.
\[
e_{j_{t-1} + 1, j_t}^{t_{\circ, j_{t-1}}} \cdots e_{j_2 + 1, j_t}^{t_{\circ, j_2}}e_{j_1 + 1, j_t}^{t_{\circ, j_1}}     f_{\beta_\bp}^{|\bp|}.v_\lambda = 0.
\]
Expanding this gives
\begin{eqnarray}\label{one}
f_{i_1, j_1}^{t_{\circ, j_1}} f_{i_1, j_2}^{t_{\circ, j_2}} \cdots f_{i_1, j_{t-1}}^{t_{\circ, j_{t-1}}} f_{i_1, j_t}^{t_{\circ, j_t}}. v_\lambda = 0.
\end{eqnarray}
We apply 
\begin{eqnarray}\label{two}
e_{i_1, i_2-1}^{t_{i_2, \circ}}\cdots e_{i_1, i_{s-1}-1}^{t_{i_{s-1},\circ}} e_{i_1, i_s-1}^{t_{i,s, \circ}}
\end{eqnarray}
to  \eqref{one} and obtain
\[
 f^{\bt}.v_\lambda = \sum_{\bs} c_{\bs} f^{\bs}.v_\lambda.
\]
with $\bs \in (\bz_{\geq 0})^{\sharp R^+}$ and finitely many $c_\bs \neq 0$. First of all, we have to show that $c_{\bs} = 0$ if $s_\alpha \neq 0$ for some $\alpha \notin A$.  This is clear up to the monomial in \eqref{one}, since $A$ is triangular and hence \eqref{rootA} implies that
 $\{ \alpha_{i_1, j_1}, \ldots, \alpha_{i_1, j_t} \} \subset A$. \\
The operators $e_{i_1, i_k -1}$ in \eqref{two} act only on monomials of the form $f_{i_1, j_\ell}$ (since $i_1$ is less than all other indices). Since
\[
e_{i_1, i_k -1}f_{i_1, j_\ell} = f_{i_k, j_\ell} \text{ or }0 \text{ if } i_k > j_\ell,
\] 
we can again deduce from \eqref{rootA}  that $\alpha_{i_k, j_\ell} \in A$ for $i_k \leq j_\ell$. This implies
\[
c_{\bs} \neq 0 \Rightarrow \bs \text{ is in the hyperplane } (\mathbb{R}_{\geq 0})^{\sharp A}\subset (\mathbb{R}_{\geq 0})^{\sharp R^+}.
\]
Using the same considerations as in \cite{FFoL11a} (we can restrict the computations there to paths supported on $A$ only) implies
\[
\mathbf{t} \preceq \mathbf{s} \Rightarrow c_{\mathbf{s}} = 0.
\]
\end{proof}

\begin{prop}\label{spaset}
Let $\lambda \in P^+$, $A$ triangular, then the set
\[
\{ f^{\bs}.v_\lambda \; | \; \bs \in S_A(\lambda) \}
\]
is a spanning set for $V_A(\lambda)^a$.
\end{prop}
\begin{proof}
Let $\bt \in (\bz_{\geq 0})^{\sharp A} \setminus S_A(\lambda)$, then there exists $\bp \in \mD_A$ such that 
\[
\sum_{\alpha \in \bp} t_\alpha > \lambda(h_{\beta_{\bp}}).
\]
Using Lemma~\ref{straight} we know that $f^{\bt}.v_\lambda \in V_A(\lambda)$ can be rewritten as
\[
\sum_{\bs \prec t} c_{\bs} f^{\bs}.v_\lambda,
\]
with $\mathbf{s} \in (\mathbb{R}_{\geq 0})^{\sharp A}$. \begin{scriptsize}
\begin{footnotesize}
•
\end{footnotesize}
\end{scriptsize}After finitely many steps we obtain $\bs \in S_A(\lambda)$ for all $c_{\bs} \neq 0$.
\end{proof}


\subsection{Linear independence}
We have to show that
\[
 \{ f^{\bs}.v_\lambda \in V_A(\lambda)^a \, : \, \bs \in S_A(\lambda) \}
\]
is linear independent in $V_A(\lambda)^a$. Since $V_A(\lambda) \subseteq V(\lambda)$ and $V_A(\lambda)_s\subset V(\lambda)_s$ it is enough to show that 
\[
 \{ f^{\bs}.v_\lambda \in V(\lambda)^a \, : \, \bs \in S_A(\lambda) \}
\]
is linear independent. For this we use  Proposition~\ref{prop-embed}, $S_A(\lambda) \subset S(\lambda)$ in combination with 
\begin{thm*}[\cite{FFoL11a}] Let $\lambda \in P^+$, then
the set
\[
 \{ \prod_{\alpha \in R^+} f_\alpha^{s_\alpha}.v_\lambda \, : \, \mathbf{s} \in S(\lambda) \}
\]
is linear independent in $V(\lambda)^a$.
\end{thm*}
So we can deduce immediately
\begin{cor}\label{linset}
 Let $\lambda \in P^+$ and $A \subset R^+$ triangular, then
the set
\[
 \{f^{\bs }.v_\lambda \, : \, \mathbf{s} \in S_A(\lambda) \}
\]
is linear independent in $V_A(\lambda)^a$.
\end{cor}

\begin{rem}
An alternative proof would be following \cite{FFoL13a}, introducing an even finer filtration to one-dimensional graded components and using that we already know that $S_A(\lambda) + S_A(\mu) = S_A(\lambda+ \mu)$. In \cite{FFoL13a} it was shown that if $S_A(\lambda)$ and $S_A(\mu)$ are both parametrizing linear independent subsets, than $S_A(\lambda + \mu)$ parametrizes a linear independent subset in the Cartan component of the tensor product. So we would be left to show that $S_A(\omega_i)$ parametrizes a linear independent subset for  all fundamental weights.
\end{rem}


\section{Degenerations}\label{sec-five}
In this section, we will see how our results combined with the results in \cite{FFoL13a} give a flat degeneration of Schubert varieties to PBW degenerated varieties and further to toric varieties. For this let us recall the notion of a favourable module, introduced in \cite{FFoL13a}. Let $\mathbb{U}$ be a complex algebraic unipotent group acting on a cyclic finite-dimensional complex vector space $M$, let $\lie n$ be the corresponding nilpotent Lie algebra, then $M = U(\lie n).m$ for some generator $m$. Let $\{f_1, \ldots, f_N \}$ be an ordered basis of $\lie n$ and fix an induced homogeneous ordering of the monomials in $U(\lie n)$. Thus we obtain a filtration of $M$ by
\[
M^{\bs} = \langle f^\bt.m \; | \; \bt \leq \bs \rangle_\bc,
\]
where in the associated graded module, $M^t$, all graded components are at most one-dimensional. Of course, $M^t$ is not a $\lie n$-module, but a $\lie n_a$-module, the abelianized version of $\lie n$. $f^{\bs}$ is called an essential monomial if $M^{\bs}/M^{ < \bs }$ is one-dimensional and we denote the finite set of the exponents of essential monomial $es(M) \subset (\bz_{\geq 0})^{N}$.
\begin{defn}
$M$ is called favourable $\mathbb{U}$-module if the following two conditions are satisfied. There exists a convex polytope $P(M) \subset (\mathbb{R}_{\geq 0})^{N}$ such that the lattice points $S(M)$ coincide with $es(M)$. For all $n \geq 1$:
\[
\dim U(\lie n).(m^{\otimes n}) = |n S(M)|,
\]
the dimension of the Cartan component in the $n$-times tensor product equals the $n$-times Minkowski sum of the lattice points in $S(M)$.
\end{defn}
Following again \cite{FFoL13a}, we denote the flag varieties associated to $M, M^a, M^t $: 
\[
\mathcal{F}_{\lie n_a}(M^a), \mathcal{F}_{\lie n_a}(M^t), \mathcal{F}_{\lie n}(M).
\]
Then it has been proved in \cite{FFoL13a} that if $M$ is a favourable $\mathbb{U}$-module, then $\mathcal{F}_{\lie n_a}(M^t)$ is a toric variety and there exist flat degenerations 
\[
\mathcal{F}_{\lie n}(M) 	\rightsquigarrow \mathcal{F}_{\lie n_a}(M^a)\rightsquigarrow \mathcal{F}_{\lie n_a}( M^t).\]
The projective flag varieties $ \mathcal{F}_{\lie n_a}(M^a), \mathcal{F}_{\lie n_a}( M^t)$ are projectively normal and arithmetically Cohen-Macaulay varieties.

\subsection{\texorpdfstring{$V_A(\lambda)$ is favourable}{Favourable in our context}}
Let $\mathbb{U} \subset SL_{n+1}$ be a complex algebraic unipotent subgroup of the lower triangular matrices, such that the corresponding Lie algebra is $\lie n_A$ for some triangular $A \subset R^+$. Further denote $B$ the standard Borel subgroup.
\begin{lem} $V_A(\lambda)$ is a favourable $\mathbb{U}$-module. Further if $w \in W$ is triangular, then $V_w(\lambda)$ is a favourable $B$-module.
\end{lem}
\begin{proof}
Let us fix an ordering on $R^+$:
\[
\alpha_{1,n} \succ \alpha_{1,n-1} \succ \alpha_{2,n} \succ \alpha_{1,n-2} \succ \alpha_{2,n-1} \succ \alpha_{3,n} \succ \ldots \succ  \alpha_1 \succ \ldots \succ \alpha_n.
\]
Note that this is different than all our previous orderings. By restriction we obtain an induced ordering on $A$.\\
It was been shown in \cite{FFoL13a}, that by choosing the ordering $\succ$ and the induced homogeneous reverse lexicographic order on monomials in $S(\lie n)$:
\[
S(\lambda) = es(V(\lambda)).
\]
Let $\bs \in (\bz_{\geq 0})^{\sharp R^+}$ such that $s_\alpha = 0$ for $\alpha \notin A$, and suppose $\bs$ is essential for $V(\lambda)$. Then 
\[
f^{\bs}.m \notin \langle f^\bt.m \, | \, \bt \prec \bs, \bt \in (\bz_{\geq 0})^{\sharp R^+} \rangle_\bc,
\]
so especially $f^{\bs}.m \notin \langle f^\bt.m \, | \, \bt \prec \bs, \bt \in (\bz_{\geq 0})^{\sharp A^+} \rangle_\bc$  (here we embed $(\bz_{\geq 0})^{\sharp A}$ trivially into $(\bz_{\geq 0})^{\sharp R^+}$). Hence $\bs$ is essential for $V_A(\lambda)$.\\
We have by Proposition~\ref{prop-embed} $S(\lambda) \cap (\bz_{\geq 0})^{\sharp A} = S_A(\lambda)$. This implies
\[
S_A(\lambda) \subset es(V_A(\lambda))
\]
but for dimension reason both set have the same cardinality and hence $S_A(\lambda) = es(V_A(\lambda))$.\\
We have further $V(\lambda + \mu) \hookrightarrow V(\lambda) \otimes V(\mu)$ as the Cartan component, hence $V_A(\lambda + \mu) \hookrightarrow V_A(\lambda) \otimes V_A(\mu)$. So 
\[
\dim U(\lie n_A^-).v_\lambda \otimes v_\mu = |S_A(\lambda + \mu)| = |S_A(\lambda) + S_A(\mu)|
\]
for all $\lambda, \mu \in P^+$, where the last equation follows from Lemma~\ref{mink-sum} . The proof is completed by considering the special case $\mu = \lambda$.
\end{proof}

\bibliographystyle{alpha}
\bibliography{pbw-demazure-bib}

\begin{thebibliography}{BBDF14}

\bibitem[ABS11]{ABS11}
Federico Ardila, Thomas Bliem, and Dido Salazar.
\newblock Gelfand-{T}setlin polytopes and
  {F}eigin-{F}ourier-{L}ittelmann-{V}inberg polytopes as marked poset
  polytopes.
\newblock In {\em 23rd {I}nternational {C}onference on {F}ormal {P}ower
  {S}eries and {A}lgebraic {C}ombinatorics ({FPSAC} 2011)}, Discrete Math.
  Theor. Comput. Sci. Proc., AO, pages 27--37. Assoc. Discrete Math. Theor.
  Comput. Sci., Nancy, 2011.

\bibitem[BBDF14]{BBCF14}
T.~Backhaus, L.~Bossinger, C.~Desczyk, and G.~Fourier.
\newblock The degree of the {H}ilbert-{P}oincar\'e polynomial of {PBW}-graded
  modules.
\newblock Preprint: arXiv:1408.0901, 2014.

\bibitem[BD14]{BD14}
T.~Backhaus and C.~Desczyk.
\newblock P{BW} filtration: {F}eigin-{F}ourier-{L}ittelmann modules via {H}asse
  diagrams.
\newblock arXiv:1407.73664, 2014.

\bibitem[BF14]{BF14}
R.~Biswal and G.~Fourier.
\newblock P{BW}-graded {D}emazure module for rectangular weight, {K}ogan faces
  and marked posets.
\newblock Preprint, 2014.

\bibitem[CF13]{CF13}
I.~Cherednik and E.~Feigin.
\newblock Extremal part of the {PBW}-filtration and {E}-polynomials.
\newblock Preprint: arXiv:1306.3146, 2013.

\bibitem[CILL14]{CLL14}
G.~Cerulli~Irelli, M.~Lanini, and P.~Littelmann.
\newblock Degenerate flag varieties and {S}chubert varieties.
\newblock Preprint, 2014.

\bibitem[CO13]{CO13}
I.~Cherednik and D.~Orr.
\newblock Nonsymmetric difference {W}hittaker functions.
\newblock Preprint: arXiv:1302.4094, 2013.

\bibitem[Fei12]{Fei12}
Evgeny Feigin.
\newblock {$\mathbb{G}_a^M$} degeneration of flag varieties.
\newblock {\em Selecta Math. (N.S.)}, 18(3):513--537, 2012.

\bibitem[FFL11a]{FFoL11a}
E.~Feigin, G.~Fourier, and P.~Littelmann.
\newblock P{BW} filtration and bases for irreducible modules in type {${A}_n$}.
\newblock {\em Transform. Groups}, 16(1):71--89, 2011.

\bibitem[FFL11b]{FFoL11b}
E.~Feigin, G.~Fourier, and P.~Littelmann.
\newblock P{BW} filtration and bases for symplectic {L}ie algebras.
\newblock {\em Int. Math. Res. Not. IMRN}, 1(24):5760--5784, 2011.

\bibitem[FFL13a]{FFoL13a}
E.~Feigin, G.~Fourier, and P.~Littelmann.
\newblock Favourable modules: Filtrations, polytopes, {N}ewton-{O}kounkov
  bodies and flat degenerations.
\newblock arXiv:1306.1292v3, 2013.

\bibitem[FFL13b]{FFoL13}
E.~Feigin, G.~Fourier, and P.~Littelmann.
\newblock P{BW}-filtration over $\mathbb{Z}$ and compatible bases for
  ${V}(\lambda)$ in type ${A}_n$ and ${C}_n$.
\newblock {\em Springer Proceedings in Mathematics and Statistics}, 40:35--63,
  2013.

\bibitem[FM14]{FM14}
E.~Feigin and I.~Makedonskyi.
\newblock Nonsymmetric {M}acdonald polynomials, {D}emazure modules and {PBW}
  filtration.
\newblock Preprint arXiv:1407.6316, 2014.

\bibitem[Fou14]{Fou14}
G.~Fourier.
\newblock New homogeneous ideals for current algebras: {F}iltrations, fusion
  products and {P}ieri rules.
\newblock Preprint: arXiv:1403.4758, 2014.

\bibitem[GJ00]{GJ00}
Ewgenij Gawrilow and Michael Joswig.
\newblock polymake: a framework for analyzing convex polytopes.
\newblock In {\em Polytopes---combinatorics and computation ({O}berwolfach,
  1997)}, volume~29 of {\em DMV Sem.}, pages 43--73. Birkh\"auser, Basel, 2000.

\bibitem[GL96]{GL96}
N.~Gonciulea and V.~Lakshmibai.
\newblock Degenerations of flag and {S}chubert varieties to toric varieties.
\newblock {\em Transform. Groups}, 1(3):215--248, 1996.

\bibitem[Gor11]{G11}
A.~Gornitsky.
\newblock Essential signatures and canonical bases in irreducible
  representations of the group ${G}_2$.
\newblock Diploma thesis, 2011.

\bibitem[HL85]{HL85}
C.~Huneke and V.~Lakshmibai.
\newblock A characterization of {K}empf varieties by means of standard
  monomials and the geometric consequences.
\newblock {\em J. Algebra}, 94(1):52--105, 1985.

\bibitem[Jos85]{Jos85}
A.~Joseph.
\newblock On the {D}emazure character formula.
\newblock {\em Ann. Sci. \'Ecole Norm. Sup. (4)}, 18(3):389--419, 1985.

\bibitem[Kog00]{Ko00}
Mikhail Kogan.
\newblock Schubert geometry of flag varieties and {G}elfand-{C}etlin theory.
\newblock {\em PhD-thesis}, 2000.

\bibitem[KST12]{KST12}
V.~A. Kirichenko, E.~Yu. Smirnov, and V.~A. Timorin.
\newblock Schubert calculus and {G}elfand-{T}setlin polytopes.
\newblock {\em Uspekhi Mat. Nauk}, 67(4(406)):89--128, 2012.

\bibitem[Sta86]{Sta86}
Richard~P. Stanley.
\newblock Two poset polytopes.
\newblock {\em Discrete Comput. Geom.}, 1(1):9--23, 1986.

\end{thebibliography}
\end{document}